\documentclass[a4paper]{amsart}

\usepackage[latin1]{inputenc}  
\usepackage{amsmath}  
\usepackage{amssymb}  
\usepackage{graphicx}   
\usepackage{caption}
\usepackage{mathrsfs}
\usepackage{euscript}
\usepackage{url}
\usepackage{pdfsync}
\usepackage{enumerate}

\newtheorem{thm}{Theorem}[section]
\newtheorem{prop}[thm]{Proposition}

\newtheorem{cor}[thm]{Corollary}
\newtheorem{lem}[thm]{Lemma}

\newtheorem{defn}[thm]{Definition}

\newtheorem{remark}[thm]{Remark}

\newcommand{\R}{{\mathbb R}}
\newcommand{\Int}{\operatorname{Int}}

 \author{Estibalitz Durand-Cartagena}
\address[E. Durand-Cartagena]
{UNED. ETSI Industriales\\
Departamento de Matemática Aplicada \\
Despacho 2.49\\
Juan del Rosal 12 \\
28040 Madrid}
\email[E. Durand-Cartagena]{edurand@ind.uned.es}

\author{Antoine Lemenant}
\address[A. Lemenant]{Université Paris 7 (Denis Diderot)\\
Laboratoire Jacques Louis Lions (CNRS UMR 7598)\\
Université Paris Diderot - Paris 7 \\
U.F.R de Mathématiques \\
Bâtiment Sophie Germain \\
75205 Paris Cedex 13 (France)}
\email[A.~Lemenant]{lemenant@ljll.univ-paris-diderot.fr}

\date{21th February, 2018}

\begin{document}

\title[Self-contracted curves]{Self-contracted curves are gradient flows of convex functions}

\maketitle

\begin{abstract}In this paper we prove that any $C^{1,\alpha}$ curve in $\R^n$, with $\alpha \in (\frac{1}{2},1]$, is the solution of the gradient flow equation for some $C^1$ convex function $f$, if and only if it is strongly self-contracted. 
\end{abstract}

\keywords 

\tableofcontents

\section{Introduction}

In this paper we are concerned with metric properties of curves that are solutions of the gradient equation of a convex function, i.e. solutions to 
\begin{eqnarray}
\gamma'(t)=-\nabla f(\gamma(t)) \label{equation}
\end{eqnarray}
for some $f:\R^n\to \R$, convex. In  \cite{DLS} it was noticed that  such curves satisfy a certain nice property called \emph{self-contractedness}, of purely metric nature: a curve $\gamma:[0,T]\to \R^n$ is self-contracted if, 
$$\|\gamma(t_2)-\gamma(t_3)\|\leq \|\gamma(t_1)-\gamma(t_3)\| \quad \forall  \,0\leq t_1 < t_2 < t_3 \leq T.$$

Here $T$ can be $+\infty$. One of the central questions in the subject that has been recently studied is to know whether self-contracted curves, if bounded, have finite length. 

This question has been positively answered in the Euclidean setting \cite{MP, DLS, DDDL} and in several non-Euclidean contexts as well \cite{scR,step,lem,O}.

In this paper we focus on the converse problem: given a self-contracted curve $\gamma$ in $\R^n$, can we find a convex function $f$ for which $\gamma$ is the solution of the gradient flow equation \eqref{equation}?

We easily see different obstructions for this to be true in its full generality. For instance, if $\gamma$ is a solution for \eqref{equation} with $f$ convex, then any reparameterization $\tilde \gamma$  of $\gamma$ would solve a different equation of the type $\tilde{\gamma}'(t)=-  m(t)\nabla f(\tilde \gamma(t))$, while $\tilde{\gamma}$ remains self-contracted. Thus, the real question is whether one can find a suitable parameterization of the given self-contracted curve, which solves \eqref{equation}.

A second remark is that any solution $\gamma$ of equation \eqref{equation} actually satisfies a more restrictive condition than self-contractedness, called \emph{strongly} self-contractedness. Indeed, for a differentiable curve, one can derive an equivalent  condition to self-contractedness by saying that the vector $-\gamma'(t)$ lies in the normal cone of the convex hull of $\gamma([t,T])$ at the point $\gamma(t)$, for any $t$. Or equivalently,
$$\langle \gamma'(t), \gamma(s)-\gamma(t)\rangle\geq 0 \quad \forall \,0\leq  t< s \leq T. $$
\begin{defn}\em
A differentiable curve $\gamma:[0,T]\rightarrow
{\mathbb{R}}^{n}$ is {\em strongly self-contracted} if for every $t,s\in [0,T]$ with $t<s$ and $\gamma'(t)\not =0$ we have that
$$
\big\langle\gamma^{\prime}(t),\gamma(s)-\gamma(t)\big\rangle>0.
$$ 
\end{defn}

As we will explain, there is no chance for a general self-contracted curve to be the solution of \eqref{equation}, unless it is \emph{strongly self-contracted}.  Furthermore, for $C^{1,\alpha}$ curves (for $\alpha \in (\frac{1}{2},1]$) the converse is also true. More precisely, here is our main result.

\begin{thm}\label{theorem} 
Let  $\gamma:[0,L] \to \R^n$ be an arc-length parameterized curve in $C^{1,\alpha}([0,L])$, with $\alpha \in (\frac{1}{2},1]$.   Then $\gamma$ is strongly self-contracted if and only if  there exist a parameterization $\tilde\gamma:[0,T]\to \R^n$  of $\gamma$ with $T<+\infty$, and a $C^1$ convex 
function $f:\R^n\to \R$, such that $$\tilde{\gamma}'(t)=-\nabla f(\tilde \gamma(t)).$$
 If $\gamma$ is  moreover $C^{3}([0,L])$  then we can find some $f$ as above 
 attaining its minimum at $\gamma(L)$ and $T=+\infty$ in this case.
\end{thm}

 Theorem \ref{theorem} follows from different intermediate statements  that we preferred to keep separated in the paper. To be more precise, the first part of the statement of Theorem \ref{theorem} comes from gathering together Lemma \ref{keyLemma1} and Lemma \ref{reductionLemma}, while the last part of the statement is a corollary of Proposition \ref{mainTh2}.

 Of course, as it was already known that solutions to the equation \eqref{equation} with $f$ convex are self-contracted curves, the most interesting part of our  result is the converse, namely, that given a self-contracted curve we can reconstruct the convex gradient flow. We shall give later some ideas about the proof.

It is worth to point out that we need some $C^{1,\alpha}$ or $C^{3}$ regularity assumption on the curve $\gamma$, while  the convex function  $f$ is only assumed to belong to $C^1$. On the other hand, it is possible to state a result in the class of $C^1$ curves if we assume a further technical condition on  the quantity $\big\langle\gamma^{\prime}(t),\gamma(s)-\gamma(t)\big\rangle$  (see Remark \ref{reguassumptions}), which is satisfied for regular enough curves. We shall give later in Section \ref{explanation} some more explanations about where exactly our regularity assumptions $C^{1,\alpha}$ or $C^{3}$ play a role.

Also, there exist some logarithmic spirals that are strongly self-contracted (see \cite{MP} and Remark \ref{spirals}). It  would  therefore be  interesting to obtain a similar result for $C^1([0,L))$ arc-length parameterized curves, instead of $C^1([0,L])$, in order to include strongly self-contracted spirals. This would be an alternative way of proving the same result as in  \cite[7.2]{DLS}, where the authors exhibit the construction of a convex function for which orbits to the associated gradient flow can spiral and, as a byproduct, prove that the so-called Thom conjecture fails for convex functions (see also Remark \ref{spirals} for further details).

Our  main result  has also some applications to convex foliation orbits. 
Recall that a direct consequence of equation \eqref{equation} is that any $C^1$ solution is transversal to the level sets of $f$, in the sense that $-\gamma'(t)$ belongs to the normal cone at point $\gamma(t)$ of $\{f=f(\gamma(t))\}$, since $\nabla f(x)$ has this property. This induces the following two definitions.

\begin{defn}[Convex foliation] \em 
A collection $\{C_r\}_{r\geq 0}$ of nonempty convex compact subsets of $\R^n$
 is a {\em (global) convex foliation} of $\R^n$ if
 $$
 r_1<r_2 \Rightarrow C_{r_1}\subset\Int C_{r_2},
 $$
 $$\bigcup_{r\geq 0} \partial C_r = \R^n\setminus \Int C_0.
 $$
 Here $\mathrm{Int}(C)$ denotes the \em interior \em of a set $C$ and $\partial C$ its boundary.
 \end{defn}

\begin{defn}[Orbit of a convex foliation]\em 
We say that  $\gamma \in C^1([0,T] ; \R^n)$ is an {\em orbit of the convex foliation} $\{C_r\}_{r\geq 0}$ if the following holds
\begin{itemize}
\item   $\gamma([0,T])\cap \Int C_0=\emptyset$. 
\item  Denoting by $r(t)$ the unique number such that $\gamma(t)\in \partial C_{r(t)}$ we have that $t\to r(t)$ is decreasing on $[0,T]$.
\item $-\gamma'(t)$  belongs to the normal cone of $C_{r(t)}$ at point $\gamma(t)$ for all $t\in [0,L].$
 \end{itemize}
 \end{defn}

The connection between self-contracted curves and orbits of a convex foliation is given by the following result.

\begin{thm}{\rm \cite[Corollary 4.11]{DDDL}} \label{thmAris}Every $C^1$ convex foliation orbit with no stationary points is a strongly self-contracted curve.
\end{thm}

Since any solution to \eqref{equation} with $f$ being convex is, in particular, the orbit of the convex foliation given by the level sets of $f$ (our function $f$ will be indeed a proper  convex function), as a byproduct of our main result we get the converse statement of Theorem \ref{thmAris} in the $C^{1,\alpha}$ class.

\begin{cor} 
Any $C^{1,\alpha}$ strongly self-contracted curve with $\alpha>1/2$ is a convex foliation orbit.
\end{cor}

But we also know  more: a consequence of our main result is that any $C^{1,\alpha}$ convex foliation orbit (with $\alpha>1/2$) is actually the gradient flow of some convex function, even if the convex foliation given by the level sets of this new convex function may not coincide with the original convex foliation. To emphasize the curiosity of the latter, let us give a last statement which directly follows from gathering together Theorem~\ref{thmAris} and \cite[Example 4.2 (ii)]{DDDL} with our main result, and which says that any  gradient flow associated to a quasiconvex function, is the gradient flow of some (other) convex function, up to reparameterization.

\begin{cor} Let $\alpha>1/2$ and $\gamma \in C^{1,\alpha}([0,T])$ be a curve which is a solution to \eqref{equation} with $f:\R^n\to \R$ a (coercive) quasiconvex function admitting a unique minimum.  Then there exists some $\tilde f:\R^n\to \R$ convex, and a new parameterization $\tilde \gamma$ of $\gamma$ such that $\tilde \gamma$ solves \eqref{equation} with $\tilde f$.
\end{cor}

We now end this introduction by giving our strategy of proof for the main result (Theorem \ref{theorem}).

\subsection{Ideas of the proof} \label{explanation}To prove our main result, we start from a given self-contracted curve, and we need to construct a convex function for which a certain parameterization of that curve satisfies $\gamma'(t)=-\nabla f(\gamma(t))$. If this holds true, since $f(\gamma(t))'=\langle \nabla f(\gamma(t)),  \gamma'(t) \rangle= -\|\gamma'(t)\|^2$, we deduce that, assuming $f$ is zero at the tip, the restriction of $f$ to the curve must be given by
\begin{eqnarray}
f(\gamma(t))=\int_t^T \|\gamma'(\tau)\|^2 d\tau. \label{defFunct}
\end{eqnarray}
We are therefore left to extend the function $f$ given on the curve by this formula, to the whole space $\R^n$. More precisely, we need to find a suitable parameterization $\gamma$ for which this can be done. We also need that the gradient of this extended function, coincides with the vector $-\gamma'(t)$ on the curve.

To do so our main tool is a  recent Theorem by Azagra and Mudarra in \cite{AM}, which precisely gives an if and only if condition under which, given a function $f$ and a mapping $G$ on a given compact set $C$, one can extend $f$ in a $C^1$ convex manner to the whole space $\R^n$ with the property that $\nabla f=G$ on the original set $C$.  
 
Let us stress that there is also a $C^{1,1}$ version of this result  and in this case a simpler explicit formula for the convex extension can be obtained (see \cite{DLS0, AM2}). The $C^{1}$ case is more delicate and the proof relies on a non trivial variant of the Whitney extension theorem. Here is the precise statement.

 \begin{thm}{\rm \textbf{\cite[Theorem 1.8}]{AM}}\label{ThAM}
 Let $C$ be a compact (non necessarily convex) subset of $\R^n$. Let $f:C\to\R$ be an arbitrary function, and $G:C\to\R^n$ be a continuous mapping. Then $f$ has a convex, $C^1$ extension $F$ to all of $\R^n$, with $\nabla F=G$ on $C$, if and only if $(f,G)$ satisfy:
 \begin{itemize}
 \item $f(x)-f(y)\geq\langle G(y),x-y\rangle\,\, \text{for all}\,\,x,y\in C$.\quad $(C)$
 \item $f(x)-f(y)=\langle G(y),x-y\rangle\implies G(x)=G(y)\,\,\text{for all}\,\,x,y\in C$.\quad $(CW1)$
 \end{itemize}
 \end{thm}
 
Our main result then follows if we can find a good choice of parameterization $\gamma$ for which the function $f$ given by \eqref{defFunct} coupled with the choice $G( \gamma(t))=-\gamma'(t)$, satisfies conditions $(C)$ and $(CW1)$ of the above theorem. 

Passing trough the change of reparameterization, this is actually equivalent to find some $m:[0,L]\to \R$ such that the arc-length parameterization $\gamma$ satisfies the modified equation $\gamma'(t)=-m(t)\nabla f(\gamma(t))$. Thus reasoning now on the fixed arc-length parameterization $\gamma$, the function $G$ is now $G(\gamma(t))=-\frac{\gamma'(t)}{m(t)}$, and the function $f$ is now $f(\gamma(t))=\int_{t}^{L}\frac{1}{m(t)}d\tau$, and we seek for some function $m$ that would satisfy $(C)$ and $(CW1)$ for those choices of $f$ and $G$. Actually, the proof of $(CW1)$ is not a problem since it will be an empty condition in our case. Only $(C)$ has to be checked.

In the proof of  $(C)$ we have to consider two cases, when $x$ is before $y$ on the curve and vice-versa. One case  is easier than the other, does not need the self-contracted assumption and holds true for all reasonable choices of function $m$ (increasing and positive). The reverse implication is more delicate and, as proved in Section \ref{sectionProof}, reduces to find $m$ that satisfies the following crucial  family of inequalities, that we call $(M)$-inequality: 
$$(M)\hspace{2cm} m(t)\int_{t}^s \frac{1}{m(\tau)} d\tau < \langle \gamma'(t), \gamma(s)-\gamma(t)\rangle  \quad \quad \forall\, 0\leq t<s\leq L.$$
Remember that here, $\gamma$ is the arc-length parameterization, and $m$ is the unknown. The main point is that the right hand side is always strictly positive because the curve is strongly self-contracted, while the left hand side can be made as small as wanted, provided that the profile of $1/m$ decreases fast enough on the interval $(t,s)$. Thus the inequality is certainly easily satisfied  for all   $t<s$ fixed, and probably uniform for $s-t$ large enough by choosing well the function $m$, but then a problem arises at small scales, i.e. for $s-t$ small. Indeed, whatever is the choice of $m$, by the Lebesgue differentiation theorem we must have 
$$m(t)\left(\frac{1}{s-t}\int_t^s\frac{1}{m(\tau)} d\tau\right) \underset{s\to t^+}{\longrightarrow} 1,$$
and we also have that 
$$ \big\langle\gamma'(t), \frac{\gamma(s)-\gamma(t)}{s-t}\big\rangle \underset{s\to t^+}{\longrightarrow} \|\gamma'(t)\|^2=1.$$
Therefore, both terms in $(M)$ are $O(s-t)$ as $s\to t$ with both constant $1$, and a suitable choice of $m$ that satisfies $(M)$ has to take into account the higher order convergence of the above two quantities, so that the former always stays smaller than the latter.    Thus we shall use some Taylor expansions of second order to control those terms, and this is where the assumption $C^{1,\alpha}$ on the curve plays a role. The choice of $m$ is done in Section \ref{proofregular}. We find that, in the $C^{1,\alpha}$ class, an exponential reparameterization would suit, i.e. the choice $m(t)=e^{bt}$ for $b$ large enough works. 

Now if one wants $f$ to attain the minimum at the endpoint of the curve, the choice of $m$ is more involved since we need the extra condition that $\lim_{t\to L}m(t)=+\infty$ in order to have $\tilde \gamma'(t)\to 0$ as $t\to L$ as required if $\tilde \gamma'(t)=-\nabla f(\tilde \gamma(t))$ and $\nabla f=0$ at the endpoint. For that type of profile, the function $1/m$ is going to zero when tending to $L$ which makes the $(M)$-inequality more and more difficult to obtain as approaching to the tip. However, assuming some further regularity on the curve ($C^3$), we are still able to construct a suitable $m$ that works. This is done in the last section of the paper and completes the proof of our main Theorem.

 Finally, since the Theorem of Azagra and Mudarra is an if and only if, we deduce the if and only if condition in the first part   of Theorem \ref{theorem} as well. It is worth mentioning that this part of the proof only needs the curve to be $C^{1}$. This is done at the end of Section \ref{proofregular}.

\subsection{Acknowledgments}  This work was done while the first author was visiting Laboratoire Jacques-Louis Lions at Paris 7 supported by the Program ``Jos\'e Castillejo'' (grant CAS17/00060). We would like to thank the LJLL for its hospitality. We also would like to thank all the people from Paris 7 with whom we had occasion to discuss about self-contracted curves. In particular Matthieu Bonnivard and  Michael Goldman for their enthusiasm about the problem and useful ideas. We would also like to thank Aris Daniilidis for introducing us to this problem. The first author is supported by the grant MTM2015-65825-P (MINECO of Spain) and  the second by the research PGMO grant COCA from the Hadamard Foundation.

\section{Preliminary results  }

In this section we collect some technical lemmas which will be needed in the proof of Theorem \ref{theorem}.

\subsection{Reparameterizing lemma}

\begin{lem}[Compact case]\label{repar}
Let $\gamma:[0,L]\to \R^n$ be a $C^1$ curve such that $\gamma^{\prime}(t)=-m(t)\nabla f(\gamma(t))$ for all $t\in [0,L]$, where   $m:[0,L]\to\R^+$ is positive and continuous. Then, there exists a $C^1$ reparameterization $\widetilde{\gamma}$ of $\gamma$ such that $\widetilde{\gamma}^{\prime}(s)=-\nabla f(\widetilde{\gamma}(s))$ for all $s\in [0,T]$, with $T:=\int_0^Lm(s)ds$.
\end{lem}

\begin{proof}
Let $T:=\int_0^Lm(s)ds$ and let  $\theta:[0,L]\to [0,T]$ be defined through
$$\theta(t)=\int_0^tm(s)ds.$$ 
We notice that $\theta$ is of class $C^1$ and stands for a strictly monotone function satisfying $\theta'(t)={m(t)}$. Now define $\widetilde{\gamma}:[0,T]\to\R^n$ with $\widetilde{\gamma}(s)=(\gamma\circ\theta^{-1})(s)=\gamma(\theta^{-1}(s))$. Observe that
$$
(\widetilde{\gamma}(s))'=((\gamma\circ\theta^{-1})(s))'=\gamma'(\theta^{-1}(s))(\theta^{-1})'(s)=\gamma'(\theta^{-1}(s))\frac{1}{\theta'((\theta^{-1})(s))}=\gamma'(\theta^{-1}(s))\frac{1}{m(\theta^{-1}(s))}
$$
and so
$$
(\widetilde{\gamma}(s))'=\gamma'(\theta^{-1}(s))\frac{1}{m(\theta^{-1}(s))}=-m(\theta^{-1}(s))\nabla f(\gamma(\theta^{-1}(s)))\frac{1}{m(\theta^{-1}(s))},
$$
which implies $(\widetilde{\gamma}(s))'=-\nabla f(\gamma(\theta^{-1}(s)))=-\nabla f(\widetilde{\gamma}(s))$.
 \end{proof}
 
 We shall also use the following variant. The proof is similar to Lemma \ref{repar} thus is left  to the reader.
 
 \begin{lem}
 \label{repar2}
Let $\gamma:[0,L]\to \R^n$ be a curve which belongs to $C^1([0,L))$, and  such that $\gamma^{\prime}(t)=-m(t)\nabla f(\gamma(t))$ for all $t\in [0,L)$, where  $m:[0,L) \to\R^+$ is positive, continuous, and $\displaystyle{\lim_{t\to L}m(t)=+\infty}$. Then, there exists a $C^1$ reparameterization $\widetilde{\gamma}:[0,\infty)\to\R^n$ of $\gamma$ such that $\widetilde{\gamma}^{\prime}(s)=-\nabla f(\widetilde{\gamma}(s))$ for all $s\in [0,+\infty)$. Moreover, $\displaystyle{\lim_{s\to +\infty} \widetilde{\gamma}^{\prime}(s)=0}$ thus $\gamma(L)=\tilde{\gamma}(\infty)$ is a critical point for $f$ (thus a minimum if $f$ is convex).
\end{lem}

\subsection{A simple Taylor-type estimate}

In this paragraph we are interested in estimating from below the function
$$s\mapsto \langle \gamma'(t) , \gamma(s)-\gamma(t)\rangle,$$
in terms of $(s-t)$ for $s>t$. This will play a major role in our main result. We are particularly interested for small $(s-t)$. Let us begin with the following simple observation. If $\gamma:[0,L]\to \R^n$ is a $C^2$ arc-length parameterized curve, it is well known that $\langle \gamma'(t),\gamma''(t)\rangle =0$, because $\frac{d}{dt}\|\gamma'(t)\|^2=0$. As a consequence, the Taylor expansion
$$\gamma(s)-\gamma(t)=\gamma'(t)(s-t)+\frac{\gamma''(t)}{2}(s-t)^2+o((s-t)^2)$$
tells us that
$$\langle \gamma'(t), \gamma(s)-\gamma(t)\rangle= (s-t)+o((s-t)^2).$$
In other words  we ``win'' some order of $(s-t)$ in the scalar product since we directly arrive to $(s-t)^2$. The following statement says that a similar phenomenon happens for $C^{1,\alpha}$ curves.

\begin{lem} \label{TaylorGood} Let $\gamma\in C^{1,\alpha}([0,L],\R^n)$ be an arc-length parameterized curve. Then for all $0\leq t<s\leq L$ we have
$$\langle \gamma'(t), \gamma(s)-\gamma(t)\rangle \geq (s-t)-C(s-t)^{2\alpha+1}.$$
\end{lem}

\begin{proof} By differentiating
$$\frac{d}{ds}\langle \gamma'(t), \gamma(s)-\gamma(t) \rangle=\langle \gamma'(t),\gamma'(s)\rangle,$$
and we obtain for $t<s$, 
\begin{eqnarray}
\langle \gamma'(t), \gamma(s)-\gamma(t)\rangle&=&\int_t^s \langle \gamma'(t),\gamma'(\tau)\rangle d\tau .\notag 
\end{eqnarray}
Then we use the formula
$$\langle \gamma'(t),\gamma'(\tau) \rangle=1-\frac{\|\gamma'(t)-\gamma'(\tau)\|^2}{2},$$
to get 
\begin{eqnarray}
\langle \gamma'(t), \gamma(s)-\gamma(t)\rangle&=&(s-t)- \frac{1}{2}\int_t^s  \|\gamma'(t)-\gamma'(\tau)\|^2 d\tau \notag  \\
&\geq& (s-t)-C(s-t)^{2\alpha+1},\notag
\end{eqnarray}
with $C=\frac{1}{2(2\alpha+1)}\|\gamma'\|_{C^{0,\alpha}}^2$.
\end{proof}

\subsection{$C^1$ strongly self-contracted implies uniformly strongly self-contracted}

\begin{defn} \label{defUSC} \em 
An arc-length parameterized curve $\gamma:[0,L]\to \R^n$ is {\em uniformly strongly self-contracted} if there exists some $c_0>0$ such that 
$$\big\langle \gamma'(t), \frac{\gamma(s)-\gamma(t)}{s-t}\big\rangle \geq  c_0  \quad \quad \forall\, 0\leq t<s\leq L.$$
\end{defn}

\begin{lem}\label{uniformlyStrong} Let $\gamma \in C^1([0,L], \R^n)$ be a strongly self-contracted curve parameterized by arc-length. Then it is uniformly strongly self-contracted.
\end{lem}
\begin{proof} We already know by strongly self-contractedness that 
$$\big \langle\gamma'(t), \gamma(s)-\gamma(t)\big \rangle>0 \quad \quad \forall\, 0\leq t<s\leq L. $$
Now we argue by contradiction. If the lemma is not true, then there exist some sequences $t_k$ and $s_k$ such that $t_k<s_k$ and 
$$\big \langle\gamma'(t_k), \frac{\gamma(s_k)-\gamma(t_k)}{s_k-t_k} \big \rangle\to 0.$$ 
By compactness of $[0,L]$ we may assume that $t_k\to t_0$ and $s_k\to s_0$ in $[0,L]$. Then if $t_0<s_0$ we get a contradiction because passing to the limit we get (notice that $\gamma$ is assumed to be $C^1$ thus $\gamma'$ is continuous), 
$$\langle\gamma'(t_k), \frac{\gamma(s_k)-\gamma(t_k)}{s_k-t_k}\rangle\to \langle\gamma'(t_0), \frac{\gamma(s_0)-\gamma(t_0)}{s_0-t_0}\rangle>0.$$
So the only possibility is that $t_0=s_0$. But then 
$$\frac{\gamma(s_k)-\gamma(t_k)}{s_k-t_k} \to \gamma'(t_0)$$
thus
$$\langle\gamma'(t_k), \frac{\gamma(s_k)-\gamma(t_k)}{s_k-t_k}\rangle \to \|\gamma'(t_0)\|^2=1\not = 0,$$
which is again a contradiction.
\end{proof}

\section{The key Lemma}
\label{sectionProof}

Our main results follow from the following Lemma. 

\begin{lem}[Key Lemma, compact case] \label{reductionLemma} 
Let $\gamma:[0,L]\to\R^n$ be an arc-length parameterized curve which belongs to $C^1([0,L])$. Let  $m:[0,L]\to \R^+_{*}$ be a continuous, increasing and positive function satisfying
 $$(M)\quad \quad \quad m(t)\int_{t}^s \frac{1}{m(\tau)}d\tau < \langle \gamma'(t) , \gamma(s)-\gamma(t)\rangle \quad \forall \,0\leq t<s\leq L. $$

Then there exists a $C^1$ convex function $f:\R^n\to \R$ such that the parameterization $\tilde\gamma:=\gamma\circ \theta^{-1}$, where $\theta(t)=\int_0^tm(s)ds$ satisfies $\tilde{\gamma}'(t)=-\nabla f(\tilde \gamma(t))$ for all $t\in [0,T]$, with $T:=\int_0^Lm(s)ds$.
\end{lem}

\begin{proof} 
The idea of the proof is to apply Theorem \ref{ThAM} to the compact set $C=\gamma([0,L])$. In particular, $C$ is the support of some $C^1$ curve, of finite length $L$. 

The strategy is to prove that
\begin{equation} \label{GradEq2}
\gamma^{\prime}(t)=-m(t)\nabla f(\gamma(t)), \quad \quad \forall t \in [0,L],
\end{equation}
and then apply Lemma \ref{repar} to find a reparameterization $\widetilde{\gamma}$ of $\gamma$ such that \eqref{equation} is verified.
Notice that if $\gamma$  satisfies \eqref{GradEq2}, then 
$$
(f(\gamma(t)))'=\langle\nabla f(\gamma(t)), \gamma'(t)\rangle =\frac{-\|\gamma'(t)\|^2}{m(t)}=\frac{1}{m(t)}.
$$

As a consequence,  if we assume $f(\gamma(L))=0$, we can integrate the above identity which yields the following necessary expression of the function $f$ restricted to  the curve, 
\begin{equation}\label{deff}
f(\gamma(t))=\int_t^{L}\frac{1}{m(\tau)} d\tau.
\end{equation}
We then define
$$
G(\gamma(t))=-\dfrac{\gamma'(t)}{m(t)},
$$
which is by hypothesis continuous on $C$.

We want now to apply Theorem \ref{ThAM} with this choice of $f$ and $G$, defined on the curve. This would say that $f$ is the restriction of some $C^1$ convex function   and $\nabla f(\gamma(t))=G(\gamma(t))=-\dfrac{\gamma'(t)}{m(t)}$ so that \eqref{GradEq2} is verified. 

We are therefore reduced to check the conditions $(C)$ and $(CW1)$ of Theorem \ref{ThAM} with our choice of $f$ and $G$. We divide the problem into two parts, for $x =\gamma(t)$ and $y=\gamma(s)$ with $t<s$ and vice-versa. 

{\bf Step 1.} Condition $(C)$ and $(CW1)$ for $x=\gamma(t)$ and $y=\gamma(s)$, with $t<s$.

We write directly
\begin{equation*}
\begin{split}
f(\gamma(t))-f(\gamma(s))=&\int_t^{L}\frac{1}{m(\tau)}d\tau-\int_s^{L}\frac{1}{m(\tau)}d\tau=\int_t^{s}\frac{1}{m(\tau)}d\tau\\
\geq& \frac{1}{m(s)} \int_t^{s}d\tau=\frac{1}{m(s)}\ell(\gamma_{|[t,s]})\quad\quad (t\mapsto \frac{1}{m(t)} \text{ strictly decreasing})\\
\geq& \frac{1}{m(s)}\|\gamma(t)-\gamma(s)\|=\frac{1}{\|m(s)\|}\|\gamma(t)-\gamma(s)\| \quad \quad  \text{ (if $m(s)>0$)}\\
\geq & \langle \frac{-\gamma'(s)}{m(s)},\gamma(t)-\gamma(s)\rangle \quad \quad  \quad \quad \text{ (by Cauchy-Schwarz)},
\end{split}
\end{equation*}
which proves $(C)$. At this point, notice that equality never occurs due to the strict inequality on the second line, so that $(CW1)$ also always holds true.

{\bf Step 2.} As we always prefer to keep the notation $x=\gamma(t)$ and $y=\gamma(s)$, with $t<s$, we now check conditions $(C)$ and $(CW1)$ with $x$ replaced by $y$, which reads 
$$
f(y)-f(x)\geq\langle G(x),y-x\rangle.
$$
Thus condition $(C)$ reduces to find $m$ such that for any $t\in I$ and any $s>t$
\begin{eqnarray}
-\int_t^s\frac{1}{m(\tau)} d\tau \geq \left\langle  -\frac{\gamma'(t)}{m(t)}, \gamma(s)-\gamma(t)\right\rangle, \notag
\end{eqnarray}
or equivalently,
\begin{eqnarray}
m(t)\int_t^s\frac{1}{m(\tau)} d\tau \leq \langle  \gamma'(t), \gamma(s)-\gamma(t)\rangle \quad \forall\, t<s, \notag
\end{eqnarray}
and this follows from our assumption called the $(M)$-inequality thus $(C)$ holds true. Moreover, since by assumption the inequality is strict, condition $(CW1)$ is empty which finishes the proof.
\end{proof}

We shall also need the following simple variant of our key Lemma.

\begin{lem}\label{reductionLemma2} Let $\gamma:[0,L]\to\R^n$ be an arc-length 
parameterized curve which belongs to $C^1([0,L))$. Let $m:[0,L)\to \R^+_{*}$ be a continuous, increasing and positive function satisfying 
\begin{eqnarray}
\lim_{t\to L}m(t)=+\infty, \label{infinity}
\end{eqnarray}
and moreover
 $$(M)\quad \quad \quad m(t)\int_{t}^s \frac{1}{m(\tau)}d\tau < \langle \gamma'(t) , \gamma(s)-\gamma(t)\rangle \quad \forall \,0\leq t<s<L. $$

Then there exists a $C^1$ convex function $f:\R^n\to \R$ such that the parameterization $\tilde\gamma:=\gamma\circ \theta^{-1}$, where $\theta(t)=\int_0^tm(s)ds$ satisfies $\tilde{\gamma}'(t)=-\nabla f(\tilde \gamma(t))$ for all $t\in [0,+\infty)$.
\end{lem}

\begin{proof}  The proof works the same way as the lemma before. The main point is to notice that, thanks to assumption \eqref{infinity}, the function 
$$
t\mapsto G(\gamma(t))=-\dfrac{\gamma'(t)}{m(t)},
$$
is continuous on $[0,L]$ (including the endpoint $L$). This allows to apply  Theorem \ref{ThAM} following the argument already used in Lemma \ref{reductionLemma}, and the reparameterization Lemma \ref{repar2} instead of Lemma \ref{repar}.
\end{proof}

\section{Proof in the $C^{1,\alpha}$ case}
\label{proofregular}

Thanks to Lemma \ref{reductionLemma}, to prove the first part of our main result  (Theorem \ref{theorem}) we are   reduced to find a suitable $m$ that satisfies the $(M)$-inequality. This is the content of the following lemma.

\begin{lem}\label{keyLemma1}  Let $\gamma:[0,L]\to\R^n$ be a strongly self-contracted curve in $C^{1,\alpha}([0,L])$ for some $\alpha\in (\frac{1}{2},1]$, parameterized by arc-length. Then the assumptions of Lemma \ref{reductionLemma} are satisfied with $m(t)=e^{b_0t}$, i.e. there exists $b_0>0$ (depending on $\gamma$) such that $m(t)=e^{b_0t}$  satisfies the $(M)$-inequality.
\end{lem}

\begin{proof}  We seek for some function $m(t)=e^{bt}$   for some parameter $b>0$. Notice that, as required in Step 1 of the proof of Lemma \ref{reductionLemma}, $m(t)>0 $ and $t \mapsto m(t)$   is increasing.

We  can directly compute
$$\int_{t}^s e^{-b\tau}d\tau =\big[\frac{1}{-b}e^{-b\tau}\big]_t^s=\frac{1}{b}\left( e^{-bt}-e^{-bs}\right).$$
By consequence, $(M)$-inequality now becomes
\begin{eqnarray}
\frac{1}{b}\left(1-e^{-b(s-t)}\right)< \langle \gamma'(t),\gamma(s)-\gamma(t)\rangle,  \quad \forall t<s\label{mainIneq_b}
\end{eqnarray}
and we are reduced to find some $b>0$ large enough so that \eqref{mainIneq_b} holds.

For this purpose we consider two cases.

{\bf Case 1}.  For $(s-t)> \frac{1}{c_0b}$. We notice that $(M)$-inequality is always satisfied in  this case. Indeed, by  the strongly self-contracted property and Lemma \ref{uniformlyStrong} we know that 
$$ \big\langle \gamma'(t),\frac{\gamma(s)-\gamma(t)}{s-t} \big\rangle \geq c_0>0 \quad \forall\, 0\leq t<s \leq L,$$
so we can estimate
$$\frac{1}{b}\left(1-e^{-b(s-t)}\right) \leq \frac{1}{b}< c_0 (s-t) \leq \big\langle \gamma'(t),\gamma(s)-\gamma(t) \big\rangle, $$
provided that 
$$\frac{1}{c_0 b}< (s-t).$$
We are therefore  reduced to check $(M)$-inequality when $(s-t)\leq \frac{1}{c_0 b}$.

{\bf Case 2}. For $(s-t)\leq \frac{1}{c_0b}$. In this case we use   the exact  Taylor formula  with integral rest to obtain the following inequality which holds for any $h\geq 0$,
\begin{eqnarray}
 e^{-bh}&=&1-bh+b^2\int_0^h(h-\tau)e^{-b\tau}d\tau \notag \\
 &\geq& 1-bh+b^2e^{-bh}\int_0^h(h-\tau)d\tau \notag \\
 &=&1-bh+b^2e^{-bh}\frac{h^2}{2} \notag
 \end{eqnarray}
from which we deduce 
\begin{eqnarray}
\frac{1}{b}\left(1-e^{-b(s-t)}\right) \leq (s-t)-be^{-b(s-t)}\frac{(s-t)^2}{2}. \label{expB} \end{eqnarray}

On the other hand since the curve is $C^{1,\alpha}$ on $[0,L]$   we can use Lemma \ref{TaylorGood} which yields,
$$\langle \gamma'(t), \gamma(s)-\gamma(t) \big\rangle \geq (s-t)-C_1(s-t)^{2\alpha+1}.$$
Returning to \eqref{expB}, we see that  \eqref{mainIneq_b} would hold true if we show that 
$$(s-t)-be^{-b(s-t)}\frac{(s-t)^2}{2}<(s-t)-C_1(s-t)^{2\alpha +1},$$
or equivalently if 
$$\frac{1}{2}be^{-b(s-t)} >C_1(s-t)^{2\alpha-1}.$$
Now we use that $2\alpha-1\geq 0$ so that 
$C_1':=C_1 ({\rm length}(\gamma))^{2\alpha-1}\geq C_1(s-t)^{2\alpha-1}$,
and what we are reduced to prove is  now
$$\frac{1}{2}be^{-b(s-t)} > C_1'.$$
Also from Case 1, we know that we need to prove the above inequality only for $(s-t)\leq \frac{1}{c_0 b}$.
In this case we have that 
$$\frac{1}{2}be^{-b(s-t)}\geq \frac{1}{2}be^{-\frac{1}{c_0}}.$$
Thus it is enough to find $b$ such that $\frac{1}{2}be^{-\frac{1}{c_0}}> C_1'$ and for instance the choice $b= 3C_1'e^{\frac{1}{c_0}}$ works.
\end{proof}

We can now give our proof of Theorem \ref{theorem}, first part.

\begin{proof}[Proof of Theorem \ref{theorem}, first part]  The ``if'' case directly follows from Lemma \ref{keyLemma1} and Lemma \ref{reductionLemma} thus we only need  to prove  the ``only if'' part. For this purpose we assume there exists a reparameterization $\widetilde{\gamma}=[0,T]\to \R^n$ of $\gamma$, with no stationary points, and a convex $C^1$-function $f:\R^n\to \R$ such that $\widetilde{\gamma}^{\prime}(t)=-\nabla f( \widetilde{\gamma}(t))$ for all $t \in [0,T]$.

If this is the case, then 
$$
(f(\widetilde{\gamma}(t)))'=\langle\nabla f(\widetilde{\gamma}(t)), \widetilde{\gamma}'(t)\rangle=-\|\widetilde{\gamma}'(t)\|^2.
$$

As a consequence, we can integrate the above identity which yields the following necessary expression of the function $f$ restricted to  the curve, 
\begin{equation*}
f(\widetilde{\gamma}(t))=\int_t^{T}\|\widetilde{\gamma}'(\tau)\|^2d\tau.
\end{equation*}

Now we apply Theorem \ref{ThAM} to $C=\tilde{\gamma}([0,T])$ ($C$ is compact because $\gamma$ is bounded) and $G(\tilde{\gamma}(t))=-\tilde{\gamma}'(t)$ so we know that conditions $(C)$ and $(CW1)$ hold. Then given $x,y\in C$,
 \[
 f(y)-f(x)\geq\langle G(x),y-x\rangle.
 \]
Without loss of generality we can assume $t\leq s$, $x=\tilde{\gamma}(t)$ and $y=\tilde{\gamma}(s)$. Observe that in this case, $f(\tilde{\gamma}(t))-f(\tilde{\gamma}(s))\geq 0$ and the following inequality has to be satisfied for any $t\leq s$
\[
f(\tilde{\gamma}(t))-f(\tilde{\gamma}(s))\leq\langle -\tilde{\gamma}'(t),\tilde{\gamma}(t)-\tilde{\gamma}(s)\rangle,
\]
so $\langle -\tilde{\gamma}'(t),\tilde{\gamma}(t)-\tilde{\gamma}(s)\rangle\geq 0$.
If $\langle -\tilde{\gamma}'(t),\tilde{\gamma}(t)-\tilde{\gamma}(s)\rangle=0$, then $\int_t^{s}\|\gamma'(\tau)\|^2d\tau=0$ and so $\gamma'(\tau)=0$ for any $t\leq\tau\leq s$, but this contradicts the fact that $\gamma$ has no stationary points and so $\langle -\tilde{\gamma}'(t),\tilde{\gamma}(t)-\tilde{\gamma}(s)\rangle> 0$. Therefore $\tilde{\gamma}$ and so $\gamma$ must be strongly self-contracted.
\end{proof}

\begin{remark}\em
Another proof of the ``only if'' part can be found in \cite[Corollary 4.11]{DDDL}. Notice that for this part we only need the curve to be $C^{1}$.
\end{remark}

\section{Proof with minimum of $f$ at the endpoint}\label{spiral}

We now would like to modify the proof so that the convex function $f$ achieves a minimum at the tip $\gamma(L)$ of the curve, in particular $\nabla f(\gamma(L))=0$. To do so, we need the reparameterized curve $\tilde \gamma$ to arrive with zero speed at the tip. Translated to $m$, we need $\displaystyle{\lim_{t\to L}m(t)=+\infty}$.

We are able to construct such a function $m$, by requiring more regularity on the arc-length parameterized curve $\gamma$. Here is our  result.

\begin{prop}\label{mainTh2} 
 Let $\gamma:[0,L]\to\R^n$ be an arc-length parameterized strongly self-contracted curve in $C^{3}([0,L])$. 
 Then there exists a $C^1$ convex function $f:\R^n\to \R$ such that the reparameterized curve $\tilde \gamma:[0,+\infty]\to \R^n$ satisfies the gradient flow equation $\tilde \gamma'(t)=-\nabla f(\tilde \gamma(t))$ on $[0,+\infty]$ and $f$ admits a minimum at $\tilde \gamma(\infty)$.
\end{prop}

\begin{remark}\em Proposition \ref{mainTh2}  directly  implies the last part of Theorem \ref{theorem}.
\end{remark}

Due to the ``key Lemma'', non-compact case (Lemma \ref{reductionLemma2}),   to prove Proposition \ref{mainTh2} it is enough to find a suitable $m$ that satisfies the $(M)$-inequality. This is the content of the  following lemma. Notice that the proof is similar to the proof of Lemma \ref{keyLemma1} but we need this time to use a more complicated construction for the function  $m$ to satisfy the required condition $\displaystyle{\lim_{t\to L}m(L)}=\infty$.

\begin{lem}\label{C2alpha} Let $\gamma:[0,L]\to\R^n$ be an arc-length parameterized strongly self-contracted curve in $C^{3}([0,L])$. There exists $b>0$ (depending on $\gamma$) such that 
$$
m(t)=\frac{1}{b(L-t)}e^{b(Lt-\frac{t^2}{2})},
$$
satisfies the $(M)$-inequality.
\end{lem}

\begin{proof}  
The aim is to find some parameter $b>0$ for which the function
$$
m(t)=\frac{1}{b(L-t)}e^{b(Lt-\frac{t^2}{2})},
$$
satisfies the $(M)$-inequality, namely,
$$
(M)\quad \quad \quad m(t)\int_{t}^s \frac{1}{m(\tau)}d\tau < \langle \gamma'(t) , \gamma(s)-\gamma(t)\rangle \quad \forall \,0\leq t<s\leq L. 
$$

To lighten the notation we shall denote by $\varphi(t):=b(Lt-\frac{t^2}{2})$ in such a way that 
$$m(t)=\frac{1}{\varphi'(t)}e^{\varphi(t)}.$$
Notice that $\varphi'(t)>0$   and $\varphi(t)> 0$ on $[0,L)$.

We first compute the left hand side of $(M)$-inequality:
\begin{eqnarray}
m(t)\int_{t}^s \frac{1}{m(\tau)}d\tau&=&m(t)\int_{t}^s \varphi'(\tau) e^{-\varphi(\tau)}d\tau  \notag \\
&=& m(t) (e^{-\varphi(t)}-e^{-\varphi(s)}) \notag \\
&=&\frac{1}{b(L-t)}\left( 1-e^{-b(L(s-t)+\frac{1}{2}(t^2-s^2))}\right). \notag
\end{eqnarray}

Now we use the uniformly strong self-contracted property to notice that, inequality $(M)$ is always satisfied in the case when 
$$\frac{1}{b(L-t)}< c_0(s-t).$$
Indeed, if the above holds true then
$$m(t)\int_t^s \frac{1}{m(\tau)}d\tau=\frac{1}{b(L-t)}\left( 1-e^{-b(L(s-t)+\frac{1}{2}(t^2-s^2))}\right)\leq \frac{1}{b(L-t)}< c_0(s-t)\leq \langle \gamma'(t) , \gamma(s)-\gamma(t) \rangle ,$$
as desired.

In conclusion, we can assume in the sequel that 
\begin{eqnarray}
\frac{1}{(L-t)}\geq b c_0(s-t). \label{weassume}
\end{eqnarray}

Then we use the mean value theorem to get the following  inequality
$$1-e^{-bh}\leq bh  \; ,\quad \forall h\geq 0,$$
from which we can estimate for $h=L(s-t)+\frac{1}{2}(t^2-s^2)$
\begin{eqnarray}
\frac{1}{b(L-t)}\left( 1-e^{-b(L(s-t)+\frac{1}{2}(t^2-s^2))}\right)\leq \frac{L(s-t)+\frac{1}{2}(t^2-s^2)}{L-t}=\frac{(s-t)(L-t)-\frac{(s-t)^2}{2}}{L-t}, \label{estimate1}
 \end{eqnarray} 

which gives us
\begin{eqnarray*}
 m(t)\int_t^s \frac{1}{m(\tau)}d\tau \leq (s-t)-\frac{1}{(L-t)}\frac{(s-t)^2}{2}.  \end{eqnarray*}

On the other hand since the curve is $C^{3}$ on $[0,L]$, a Taylor expansion of $\gamma$ gives us that 
$$\langle \gamma'(t), \gamma(s)-\gamma(t) \big\rangle \geq (s-t)-C_1(s-t)^{3},$$
with $C_1$ depending only on $\|\gamma^{(3)}\|_{\infty}$.
Now $(M)$-inequality is satisfied provided
$$(s-t)-\frac{1}{(L-t)}\frac{(s-t)^2}{2}< (s-t)-C_1(s-t)^{3},$$
or equivalently if 
$$\frac{1}{(L-t)} > 2C_1(s-t).$$
But now using \eqref{weassume} it is guaranteed if we choose $b$ large enough so that $b c_0> 2C_1$. This ends the proof.
\end{proof}

\begin{remark}\em\label{reguassumptions}
By assuming a more technical condition, we can weaken the hypotheses of Lemma \ref{C2alpha}, to work in the class 
of $C^1([0,L))$ curves. Indeed, the construction of $m$ in the proof of  Lemma \ref{C2alpha} holds if we assume instead of $C^3([0,L])$ that $\gamma$ is a uniformly strongly self-contracted arc-length parameterized curve that belongs to $C^{1}([0,L))$ and that there exists a positive and continuous  increasing function $\zeta : [0,L)\to [0,+\infty]$ such that 
\begin{eqnarray}
A:=\int_0^L \zeta(\tau) d \tau < +\infty \label{zeTA1}
\end{eqnarray}
and
\begin{eqnarray}
\big\langle \gamma'(t), \frac{\gamma(s)-\gamma(t)}{s-t}\big \rangle \geq 1-\zeta(s)(s-t)^2 \quad \quad \forall \,0\leq t<s<+\infty. \label{zeTA}
\end{eqnarray}
For instance if the curve is $C^3([0,L))$ one can choose $\zeta(t):=\sup_{\tau \in [0,t]}\|\gamma^{(3)}(\tau)\|$.

In this case one can use the same proof as the one of Lemma \ref{C2alpha} with  the choice 
$$m(t)=\frac{1}{\varphi'(t)}e^{\varphi(t)},$$
and where $\varphi(t)$ is constructed as follows:
$$\varphi'(t)=\frac{2A}{c_0}-\frac{2}{c_0}\int_0^t \zeta(\tau)\,d\tau , \quad  \text{ and then }\quad \varphi(t)=\int_0^t\varphi'(\tau)\,d\tau.$$

It would be  interesting to find an explicit example of a self-contracted spiral satisfying \eqref{zeTA1} and  \eqref{zeTA}, but unfortunately, we do not know whether this would be possible. A motivation is given in the following remark.
\end{remark}

 \begin{remark}\em\label{spirals}
There exist strongly self-contracted curves that turn around its limit point infinitely many times and so $\gamma'$ does not admit a limit at the tip.
Let $\lambda\in(0,\infty)$ and consider the class of logarithmic spirals
 \[
 \gamma_{\lambda}(t)=(e^{-t\lambda}\cos t,e^{-t\lambda}\sin t)\quad \forall t\geq 0.
 \]
 It was shown in \cite{MP} that there exists a curve $\gamma_{\lambda_0}$ with 
$\lambda_0=\exp(\frac{-3\pi\lambda_0}{2})$ such that for any point $t\in I$ there exists a unique $u>t$ such that the normal to the tangent line at $\gamma_{\lambda_0}(t)$ is tangent at $\gamma_{\lambda_0}(u)$. Moreover, it is the only $C^1$ curve (up to a rigid motion) satisfying that property and whose length equals the perimeter of its convex hull. We can now classify logarithmic spirals in terms of the parameter $\lambda_0$.
 \begin{itemize}
 \item If $\lambda<\lambda_0$, $ \gamma_{\lambda}$ is not self-contracted.
 \item If $\lambda=\lambda_0$, $ \gamma_{\lambda}$ is self-contracted but not strongly self-contracted.
 \item If $\lambda>\lambda_0$, $ \gamma_{\lambda}$ is uniformly strongly self-contracted.
 \end{itemize} 
The arc-length parameterized spiral for $\lambda>\lambda_0$ is a strongly self-contracted curve that belongs to $C^{\infty}([0,L))$ but in this case our result does not apply (since we need regularity up to the tip).  In \cite[7.2]{DLS} the authors prove that there exists a convex function for which orbits to the associated gradient flow can spiral and, as a byproduct, the so-called Thom conjecture fails for convex functions. It would be   interesting to know if we can obtain Theorem \ref{theorem} for the family of strongly self-contracted spirals.
\begin{figure}[h]
\begin{center}
\includegraphics[width=4cm]{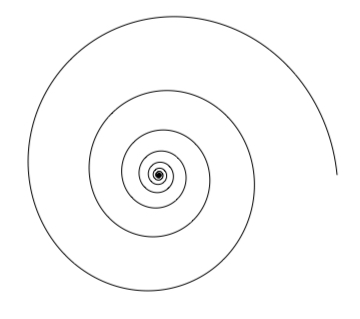}
\includegraphics[width=4cm]{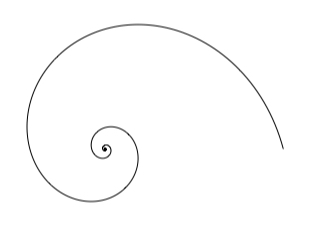}
\includegraphics[width=4cm]{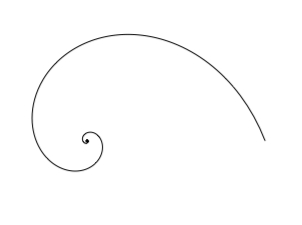}
\caption{Logarithmic spirals with $\lambda<\lambda_0$, $\lambda=\lambda_0$ and $\lambda>\lambda_0$. }\label{ejes}
\end{center}
\end{figure}
 \end{remark}

\end{document}